%% file: draft2_ver4.tex
\documentclass[12pt]{amsart}
\usepackage{amssymb,comment}

\numberwithin{equation}{section}

\def\NZQ{\mathbb}               % the font for N,Z,Q,R,C

\def\QQ{{\NZQ Q}}
\def\ZZ{{\NZQ Z}}
\def\RR{{\NZQ R}}

\def\PP{{\NZQ P}}

\def\ab{\mathbf{a}}

\def\xb{\mathbf{x}}

\def\opn#1#2{\def#1{\operatorname{#2}}} % to make operators
\opn\pd{pd} 
\opn\rk{rk}
\opn\rank{rank}
\opn\depth{depth} 
\opn\grade{grade} 
\opn\height{height}
\opn\embdim{emb\,dim} 
\opn\codim{codim}
\opn\Tr{Tr} 
\opn\bigrank{big\,rank}
\opn\lcm{lcm}
\opn\reg{reg} 
\opn\ini{in} 
\opn\size{size}
\opn\mult{mult}
\opn\dist{dist}
\opn\cone{cone}
\opn\lex{lex}
\opn\rev{rev}
\opn\div{div}
\opn\Div{Div}
\opn\cl{cl}
\opn\Cl{Cl}
\opn\Syz{Syz} \opn\Im{Im} \opn\Ker{Ker} \opn\Coker{Coker}
\opn\Hom{Hom} \opn\Tor{Tor} \opn\Ext{Ext}
\opn\End{End} \opn\Aut{Aut} \opn\id{id} \opn\nat{nat}
\opn\mod{mod} \opn\ord{ord}
\opn\aff{aff} \opn\con{conv} \opn\relint{relint} \opn\st{st}
\opn\lk{lk} \opn\cn{cn} \opn\core{core} \opn\vol{vol}
\opn\link{link} \opn\star{star} \opn\sgn{sgn}
\def\Ac{{\mathcal A}}

\def\Hc{{\mathcal H}}

\def\Fc{{\mathcal F}}
\def\Pc{{\mathcal P}}

\newtheorem{Theorem}{Theorem}[section]
\newtheorem{Lemma}[Theorem]{Lemma}
\newtheorem{Corollary}[Theorem]{Corollary}
\newtheorem{Proposition}[Theorem]{Proposition}
\newtheorem{Remark}[Theorem]{Remark}

\newtheorem{Conjecture}[Theorem]{Conjecture}

\let\epsilon\varepsilon
\let\phi=\varphi
\let\kappa=\varkappa
\def\qed{\ifhmode\textqed\fi
      \ifmmode\ifinner\quad\qedsymbol\else\dispqed\fi\fi}
\def\textqed{\unskip\nobreak\penalty50
       \hskip2em\hbox{}\nobreak\hfil\qedsymbol
       \parfillskip=0pt \finalhyphendemerits=0}
\def\dispqed{\rlap{\qquad\qedsymbol}}

\newcommand{\set}[1]{\left\{ #1 \right\}}  % set
\newcommand{\with}{\ \vrule\ }  % 'with'-Symbol in sets
\newcommand{\defa}{:=}
\newcommand{\bvec}[1]{b_{#1}}

\makeatletter
  \newcommand{\labelformat}[1]{%
    \expandafter\renewcommand\csname p@#1\endcsname[1]%
  }
  \newcommand{\numberlike}[2]{%
     \expandafter\def\csname c@#1\endcsname{\csname c@#2\endcsname}%
  }
\makeatother

\newcommand\tD{\widetilde\Delta}

\newcommand\fp{\mathfrak{p}}

\opn\Spec{Spec}
\opn\supp{supp}
\opn\ch{char}
\newcommand\setb[2]{\{ #1 : #2\} }
\newcommand\floor[1]{\lfloor #1 \rfloor}

\begin{document}
\title{Toric rings arising from cyclic polytopes}

\author{Takayuki Hibi}
\address{Department of Pure and Applied Mathematics, Graduate School
of Information Science and Technology, Osaka University, Toyonaka, Osaka 560-
0043, Japan}%
\email{hibi@math.sci.osaka-u.ac.jp}%
% \thanks{JST CREST}
%
\author{Akihiro Higashitani}
\address{Department of Pure and Applied Mathematics, Graduate School
of Information Science and Technology, Osaka University, Toyonaka, Osaka 560-
0043, Japan}%
\email{a-higashitani@cr.math.sci.osaka-u.ac.jp}%
% \thanks{Supported by JSPS Research Fellowship for Young Scientists} 
%
\author{Lukas Katth\"an}
\address{Fachbereich Mathematik und Informatik, Philipps-Universit\"at Marburg, 
35032 Marburg, Germany 
% \newline
% Department of Pure and Applied Mathematics, Graduate School of Information Science and Technology, 
% Osaka University, Toyonaka, Osaka 560-0043, Japan
}
\email{katthaen@mathematik.uni-marburg.de}%
% \thanks{Supported by DAAD}
%
\author{Ryota Okazaki}
\address{Department of Pure and Applied Mathematics, Graduate School
of Information Science and Technology, Osaka University, Toyonaka, Osaka 560-0043, Japan}%
\email{okazaki@ist.math.sci.osaka-u.ac.jp}%
% \thanks{JST CREST}
\thanks{
{\bf 2010 Mathematics Subject Classification:}
Primary 13H10; Secondary 52B20, 05E40. \\
\hspace{5.3mm}{\bf Key words and phrases:} cyclic polytope, Serre's condition $(R_1)$, 
Cohen--Macaulay, Gorenstein. \\
\hspace{5.3mm}The first and forth authors are supported by the JST CREST 
``Harmony of Gr\"obner Bases and the Modern Industrial Society.'' \\
\hspace{5.3mm}The second author is supported by JSPS Research Fellowship for Young Scientists. \\
\hspace{5.3mm}This research was performed while the third author was staying
at Department of Pure and Applied Mathematics, 
Osaka University, November 2011 -- April 2012, supported by the DAAD
}

\begin{abstract}
Let $d$ and $n$ be positive integers with $n \geq d+1$ and 
$\Pc \subset \RR^d$ an integral cyclic polytope of dimension $d$ with $n$ vertices and 
let $K[\Pc]=K[\ZZ_{\geq 0}\Ac_{\Pc}]$ denote its associated semigroup $K$-algebra, 
where $\Ac_{\Pc}=\set{(1,\alpha) \in \RR^{d+1} : \alpha \in \Pc} \cap \ZZ^{d+1}$ and $K$ is a field. 
In the present paper, we consider the problem 
when $K[\Pc]$ is Cohen--Macaulay by discussing Serre's condition $(R_1)$ and 
we give a complete characterization when $K[\Pc]$ is Gorenstein. 
Moreover, we study the normality of the other semigroup $K$-algebra $K[Q]$ 
arising from an integral cyclic polytope, 
where $Q$ is a semigroup generated only with its vertices. 
\end{abstract}

\maketitle

\section*{Introduction}

Let $d$ and $n$ be positive integers with $n \geq d+1$ and 
$\tau_1, \ldots, \tau_n$ real numbers with $\tau_1 < \cdots < \tau_n$. 
The convex polytope $C_d(\tau_1,\ldots,\tau_n) \subset \RR^d$ which is the convex hull of 
$\set{(\tau_i,\tau_i^2,\ldots,\tau_i^d) \in \ZZ^d : i=1,\ldots,n}$ is called 
a {\em cyclic polytope} of dimension $d$ with $n$ vertices. 
In particular, when $\tau_1, \ldots, \tau_n$ are integers, we call it 
an {\em integral} cyclic polytope.
A cyclic polytope is one of the most significant polytopes,
and is well known to be the polytope giving the upper bound
in The Upper Bound Theorem due to Motzkin and McMullen (cf. \cite{BH,Gru}).
Several studies on cyclic polytopes have been achieved by many researchers,
but there are only a few studies on toric rings of integral cyclic polytopes.
In the previous paper \cite{HHKO}, we discussed the normality of (toric rings of) cyclic polytopes
(see below for the definition of normality of a polytope),
and gave a sufficient condition (\cite[Theorem 2.1]{HHKO}) 
and a necessary one (\cite[Theorem 3.1]{HHKO}) for $\Pc$ to be normal.
The present paper is devoted to the continuation of the study of $\Pc$.
We also study the semigroup $K$-algebra generated only by the vertices of integral cyclic polytopes.

For an integral convex polytope $\Pc \subset \RR^{N}$, that is, a convex polytope
whose vertices are in $\ZZ^N$, define $\Pc^{*} \subset \RR^{N+1}$ to be the convex hull of all points 
$(1, \alpha) \in \RR^{N+1}$ with $\alpha \in \Pc$, and set $\Ac_\Pc = \Pc^{*} \cap \ZZ^{N+1}$.
Note that $\Ac_\Pc$ is just the set of integer points in $\Pc^{*}$. 
Let $\ZZ_{\geq 0}$ denote the set of nonnegative integers and 
$\RR_{\geq 0}$ that of nonnegative real numbers. We say that $\Pc$ is {\em normal} if one has 
$$\ZZ_{\geq 0} \Ac_{\Pc}=\ZZ \Ac_{\Pc} \cap \RR_{\geq 0} \Ac_{\Pc},$$
where $\ZZ_{\geq 0}\Ac_{\Pc}, \ZZ \Ac_{\Pc}, \RR_{\ge 0} \Ac_{\Pc}$ are the set of the linear combinations of $\Ac_\Pc$
(in $\RR^{N+1}$) with the coefficients in $\ZZ_{\ge 0}, \ZZ, \RR_{\ge 0}$, respectively.

Let $K$ be a field $\Pc \subset \RR^d$ an integral cyclic polytope of dimension $d$. 
%We define two semigroup $K$-algebras arising from an integral cyclic polytope $\Pc=C_d(\tau_1,\ldots,\tau_n)$.
We refer the reader to \cite{brunsgubel,BH} for the definition of an affine semigroup
and its affine semigroup $K$-algebra.
Let $\ZZ_{\ge 0}\Ac_\Pc$ be as above,
and set
$$
Q_d(\tau_1,\dots ,\tau_n)=\ZZ_{\geq 0}\set{(1,\tau_i,\tau_i^2,\ldots,\tau_i^d) \in \ZZ^{d+1} : i=1,\ldots,n}.
$$
Both of $\ZZ_{\ge 0}\Ac_\Pc$ and $Q_d(\tau_1,\dots ,\tau_n)$ are affine semigroups contained in $\ZZ^{d+1}$;
The further is generated by the set of integer points in $\Pc^{*}$,
and the latter only by the vertices of $\Pc^{*}$.
For simplicity, set $Q := Q_d(\tau_1,\cdots ,\tau_n)$.
Following usual convention, let $K[\Pc]$ denote the affine semigroup $K$-algebra
of $\ZZ_{\ge 0}\Ac_\Pc$, and let $K[Q]$ be that of $Q$.
The $K$-algebra $K[\Pc]$ is just the $K$-subalgebra of the polynomial ring $K[t_0,t_1,\ldots,t_d]$
such that
$$
K[\Pc] = \bigoplus_{\ab \in \ZZ_{\geq 0} \Ac_{\Pc}} K \cdot t^\ab,
$$
where we set $t^\ab=t_0^{a_0}t_1^{a_1} \cdots t_d^{a_d}$ for $\ab=(a_0,a_1,\ldots,a_d) \in \ZZ_{\geq 0}^{d+1}$.
Note that $K[\Pc]$ is nothing other than the {\em toric ring} of $\Pc$,
and as is well known, $\Pc$ is normal if and only if so is $K[\Pc]$. 

Similarly, $K[Q]$ is the $K$-subalgebra of $K[t_0,t_1,\ldots,t_d]$ with $K[Q] = \bigoplus_{\ab \in Q} K \cdot t^{\ab}$.
It is just the toric ring associated with the following {\em configuration}
$$
\begin{pmatrix}
1        &1        &\cdots  &1 \\
\tau_1   &\tau_2   &\cdots  &\tau_n   \\
\tau_1^2 &\tau_2^2 &\cdots  &\tau_n^2 \\
\vdots   &\vdots   &        &\vdots   \\
\tau_1^d &\tau_2^d &\cdots  &\tau_n^d 
\end{pmatrix}.
$$

\bigskip

%Let, as before, $d$ and $n$ be positive integers with $n \geq d+1$, 
%$\tau_1,\ldots,\tau_n$ integers with $\tau_1 < \cdots < \tau_n$ and $\Pc=C_d(\tau_1,\ldots,\tau_n)$. 
 
In the present paper, 
we will consider the Cohen-Macaulayness and Gorensteinness of $K[\Pc]$ 
(Theorem \ref{CM} and Theorem \ref{Gor}, respectively). 
We prove that $K[\Pc]$ always satisfies Serre's condition $(R_1)$, which implies that 
$K[\Pc]$ is Cohen-Macaulay if and only if it is normal. 
This means that the characterization of the normality of integral cyclic polytopes is also 
that of its Cohen-Macaulayness. 
Moreover, it will turn out that $K[\Pc]$ is Gorenstein if and only if one has 
$d=2$, $n=3$ and $(\tau_2-\tau_1, \tau_3-\tau_2)=(2,1)$ or $(1,2)$, 
which says that there is essentially only one Gorenstein integral cyclic polytope, 
% case where $K[\Pc]$ is Gorenstein is essentially only one. 
see Lemma \ref{equiv}. 
We also discuss the normality of the $K$-algebra $K[Q]$, and 
show that if $d \geq 2$ and $n=d+2$, then $K[Q]$ is not normal (Theorem \ref{thm:n=d+2 with d > 1}). 

\smallskip

The structure of the present paper is as follows. 
After preparing some notation, terminologies and lemmata in Section 1 for our main theorems, 
we show Theorem \ref{CM} in Section 2 by considering 
Serre's condition $(R_1)$ for $K[\Pc]$ (Proposition \ref{R1condition}). 
Moreover, Section 3 is devoted to proving Theorem \ref{Gor}. 
In Section 4, we study $K[Q]$ and prove Theorem \ref{thm:n=d+2 with d > 1}. 

%%%%%%%%%%%%%%%%%%%%%%%%%%%%%%%%%%%%%%%%%%%%%%%%%%%

\input{prelim_3}

\input{R1_prop_2}

\input{Gor_2}

\input{without_int_ver3}

\end{document}

%% file: prelim_3.tex
\section{Preliminaries}

In this section, we prepare notation and lemmata for our main theorems. 
Most of them are refered from \cite[Section 1]{HHKO}.

\smallskip

First of all, we will review some fundamental facts on cyclic polytopes.
Let $d$ and $n$ be positive integers with $n \geq d+1$. 
Throughout the present paper,  it is convenient and essential 
to work with a homogeneous version of the cyclic polytopes, 
hence we consider $C^*_d( \tau_1, \ldots, \tau_n)$ (or $\RR_{\geq 0} Q$) 
instead of $C_d( \tau_1, \ldots, \tau_n)$. 
For $n$ real numbers $\tau_1, \ldots, \tau_n$ with $\tau_1 < \cdots < \tau_n$, we set 
$$v_i \defa (1, \tau_i, \tau_i^2, \ldots, \tau_i^d) \in \RR^{d+1} \; \text{ for } \; 1 \leq i \leq n. $$ 
In other words, $C^{*}_d( \tau_1, \ldots, \tau_n) = \con(\{ v_i : 1 \leq i \leq n \}) \subset \RR^{d+1}$. 
See \cite[Chapter 0]{ziegler} for some basic properties of cyclic polytopes.

Recall that the cyclic polytope $C^*_d(\tau_1,\dots ,\tau_n)$ is {\em simplicial}, i.e.,
all of their faces are simplexes.
Hence its boundary complex is just a $(d-1)$-dimensional simplicial complex on $\set{v_1,\dots ,v_n}$.
Following usual convention, we set $[n] := \set{1,\dots ,n}$.
Assigning $i$ to $v_i$ for each $i$, we can regard the simplicial complex as the one on $[n]$.
Let $\Gamma_d(\tau_1,\dots ,\tau_n)$ denote this simplicial complex on $[n]$.
The faces of $\Gamma_d(\tau_1,\dots ,\tau_n)$ are completely characterized
in terms of their {\em type}.
A non-empty subset $W \subset [n]$ is said to be {\em contiguous} if
$W = \{ i,i+1,\ldots, j\}$
for some positive integers $i$ and $j$ with $1 < i  \le j  < n$,
and to be an {\em end set} if either $W = \{ 1,\dots, i\}$
or $W = \{ i,\dots ,n \}$ for some $i$ with $1 \le i \le n$.
We set
$\max \emptyset := 1$ and $\min \emptyset := n$.
Any subset $W \subset [n]$ has unique decomposition
\begin{equation}\label{eq:decomp}
W := Y_1 \sqcup X_1 \sqcup X_2 \sqcup \cdots \sqcup X_t
\sqcup Y_2,
\end{equation}
such that
\begin{enumerate}
\item $Y_1$, $Y_2$ are empty or end sets, and each $X_i$ is contiguous;
\item $\max X_i < \min X_{i+1}$ for all $i$ with $0 \le i \le t$,
where we set $X _0 := Y_1$ and $X_{t+1} := Y_2$.
\end{enumerate}
The subset $W$ is said to be of {\em type} $(r,s)$ where $r = \# W$
and $s = \# \setb{i}{\# X_i \text{ is odd}}$.

\begin{Proposition}[cf. {\cite[pp. 226--227]{BH}}]\label{thm:prop_cycl_poly}
Let $W$ be a subset of $[n]$. The following statements hold. 
\begin{enumerate}
\item Any $d+1$ elements of $v_1,\dots,v_n$ are linearly independent over $\RR$. 
\item If $\# W \le \floor{d/2}$, then $W$ is a face of 
$\Gamma_d( \tau_1, \ldots, \tau_n)$. 
\item The subset $W$ is a face of $\Gamma_d( \tau_1, \ldots, \tau_n)$ of dimension $\# W - 1$ 
if and only if $0 \le \# W \le d$ and $W$ is a type $(\# W, s)$ for some integer $s$ 
with $0 \le s \le d- \# W$.
\end{enumerate}
\end{Proposition}

Hereafter, we will assume that $\tau_1,\ldots,\tau_n$ are integers, 
that is, $v_1,\ldots,v_n \in \ZZ^{d+1}$. 
We introduce a special representation of integral cyclic polytopes which is sometimes helpful. 
Write the vectors $v_1,\ldots,v_n$ as column vectors into a matrix, namely, 
\begin{equation}\label{eq:matrix1}
(v_1,v_2,\ldots,v_n)=
\begin{pmatrix}
1        &1        &\cdots  &1 \\
\tau_1   &\tau_2   &\cdots  &\tau_n   \\
\tau_1^2 &\tau_2^2 &\cdots  &\tau_n^2 \\
\vdots   &\vdots   &        &\vdots   \\
\tau_1^d &\tau_2^d &\cdots  &\tau_n^d 
\end{pmatrix}. 
\end{equation}

\begin{Lemma}\label{eg:delta}
By elementary row operations, \eqref{eq:matrix1} is transformed into the matrix 
\begin{eqnarray}\label{eq:mat}
\begin{pmatrix}
1         & 1          & 1           & \cdots & 1                 & 1           & \cdots & 1 \\
0         & \tD_{1,2} & \tD_{1,3} & \cdots & \tD_{1,d}   & \tD_{1,d+1}  & \cdots & \tD_{1,n} \\
0         & 0          & \tD_{2,3}  & \cdots & \tD_{2,d}   & \tD_{2,d+1} & \cdots & \tD_{2,n} \\
\vdots & \vdots   & \ddots    & \ddots &                   & \vdots    & \cdots & \vdots \\
\vdots & \vdots   &              & \ddots & \ddots         &\vdots    & \cdots & \tD_{d - 2,n} \\
0         & 0          &\cdots       & \cdots & 0                 & \tD_{d,d+1} & \cdots & \tD_{d,n} 
\end{pmatrix},
\end{eqnarray}
where $\tD_{i,j} := \prod_{k=1}^i\Delta_{kj}$ and $\Delta_{ij}=\tau_j-\tau_i$ for $1 \leq i\not=j \leq n$. 
In particular, the convex hull of the column vectors of this matrix is unimodularly equivalent to $C^*_d(\tau_1, \ldots, \tau_n)$. 
\end{Lemma}

%A proof of the above lemma is essentially the same as a proof of the well-known Vandermonde determinant. 
Note that Lemma \ref{eg:delta} is valid for any ordering of the parameters $\tau_1,\dotsc,\tau_n$, 
i.e., any ordering of $v_1,\ldots,v_n$. 
By this lemma, $C_d^*(\tau_1,\ldots,\tau_n)$ can be regarded as the convex hull of 
the column vectors of \eqref{eq:mat}. 
Thus, in the sequel, we will often use $v_i$ as the $i$th column vector of \eqref{eq:mat} 
in stead of \eqref{eq:matrix1}.

We also often apply the following 
\begin{Lemma}\label{equiv}
An integral cyclic polytope $C^{*}_d(\tau_1,\ldots,\tau_d)$ is unimodularly equivalent to 
$C^{*}_d(-\tau_n, \ldots, -\tau_1)$. Moreover, for any integer $m$, 
$C^{*}_d(\tau_1,\ldots,\tau_d)$ is unimodularly equivalent to $C^{*}_d(\tau_1 + m , \ldots, \tau_n + m)$. 
\end{Lemma}

The first statement of this lemma says that 
$C_d(\tau_1,\ldots,\tau_n)$ is unimodularly equivalent to $C_d(\tau_1',\ldots,\tau_n')$ 
if $\tau_{i+1}'-\tau_i'=\Delta_{n-i,n-i+1}$ for every $1 \leq i \leq n-1$.

%
% b-vectors
We define a certain class of vectors which we will use in Section \ref{R_1}. 
Let $S=\{i_1,\ldots,i_q\} \subset [n]$ be a non-empty set, where $i_1< \cdots <i_q$. 
Then we define 
\[ \bvec{S} \defa \sum_{i\in S} \frac{1}{\prod_{j \in S \setminus \{i\}} \Delta_{ij}} v_i 
=\sum_{k=1}^q \frac{(-1)^{k+1}}{\prod_{j \in S \setminus \{i_k\}} |\Delta_{i_kj}|}v_{i_k}, \]
where $\bvec{S}=v_{i_1}$ when $q=1$, i.e., $\# S=1$. 
If $S$ is small, we will sometimes omit the brackets around the elements, 
thus we write, for example, $\bvec{ij} = \bvec{\set{i,j}}$. 
However, the vector does not depend on the order of the indices.
%
%\begin{Example}{\em 
%Let us write down $\bvec{S}$'s for small sets $S$. Assume $1\leq i<j<k<l\leq n$. Then 
%\begin{align*}
%\bvec{i} &= v_i, \\
%\bvec{ij} &= \frac{1}{\Delta_{ij}} v_i - \frac{1}{\Delta_{ij}} v_j, \\
%\bvec{ijk} &= \frac{1}{\Delta_{ij}\Delta_{ik}} v_i - \frac{1}{\Delta_{ij}\Delta_{jk}} v_j 
%+ \frac{1}{\Delta_{ik}\Delta_{jk}} v_k, \\
%\bvec{ijkl} &= \frac{1}{\Delta_{ij}\Delta_{ik}\Delta_{il}} v_i - \frac{1}{\Delta_{ij}\Delta_{jk}\Delta_{jl}} v_j 
%+ \frac{1}{\Delta_{ik}\Delta_{jk}\Delta_{kl}} v_k - \frac{1}{\Delta_{il}\Delta_{jl}\Delta_{kl}} v_l. 
%\end{align*}
%The sign changes are due to a reordering of the indices since $\Delta_{ij} = - \Delta_{ji}$. 
%If $v_i,v_j,v_k,v_l$ are given in the form \eqref{eq:mat}, i.e., if 
%\begin{equation*}
%\begin{pmatrix}
%v_i \\
%v_j \\
%v_k \\
%v_l
%\end{pmatrix}=
%\begin{pmatrix}
%1&0           &\cdots                 &\cdots                            &\cdots &\cdots &0 \\
%1&\Delta_{ij} &0                      &\ddots                            &\cdots &\cdots &\vdots \\
%1&\Delta_{ik} &\Delta_{ik}\Delta_{jk} &\ddots                            &\cdots &\cdots &\vdots \\
%1&\Delta_{il} &\Delta_{il}\Delta_{jl} &\Delta_{il}\Delta_{jl}\Delta_{kl} &0      &\cdots &0\\
%\end{pmatrix}, 
%\end{equation*}
%then $\bvec{i}=(1,0,\ldots,0)$, $\bvec{ij}=(0,-1,0,\ldots,0)$, $\bvec{ijk}=(0,0,1,0,\ldots,0)$ 
%and $\bvec{ijkl}=(0,0,0,-1,0,\ldots,0)$. 
%In general, $\bvec{1}, \bvec{12}, \ldots, \bvec{12 \cdots d+1}$ look like 
%$(0,\ldots,0,\pm 1,0,\ldots,0)$ when $v_1,\ldots,v_{d+1}$ are of the form \eqref{eq:mat}. 
%}\end{Example} 
%
The following proposition collects the basic properties on these vectors. 
\begin{Proposition}\label{prop:bvectors}
\begin{enumerate}
\item For any non-empty set $S \subset [n]$, one has $\bvec{S} \in \ZZ^{d+1}$. 
\item Let $S \subset [n]$ and $a,b\in S$ with $a\neq b$. Then we have a recursion formula 
\[ \bvec{S} = \frac{1}{\Delta_{ba}} \bvec{S\setminus \{a\}} + \frac{1}{\Delta_{ab}} \bvec{S\setminus \{b\}}. \]
\item For any distinct $d+1$ indices $i_1,\dotsc,i_{d+1} \in [n]$ (not necessarily ordered), the vectors
\[ \bvec{i_1}, \bvec{i_1 i_2}, \bvec{i_1 i_2 i_3}, \ldots, \bvec{i_1 \cdots i_{d+1}} \]
form a $\ZZ$-basis for $\ZZ^{d+1}$. 
\item If $\# S \geq d+2$, then $\bvec{S}=0$.
\end{enumerate}
\end{Proposition}
%
%
%
%
%
%

%Applying this construction, we obtain another useful fact on integral cyclic polytopes. 

Finally, by applying this construction, we state 
\begin{Lemma}\label{lemma:basis}
For an integral cyclic polytope $\Pc \subset \RR^d$ of dimension $d$, 
one has 
\[\ZZ \Ac_{\Pc}=\ZZ^{d+1}\,.\] 
%The integer points contained in an integral cyclic polytope generate $\ZZ^d$.
\end{Lemma}

%% file: R1_prop_2.tex
\section{Serre's $(R_1)$ property}\label{R_1}

In this section, we prove that $K[\Pc]$ always satisfies Serre's Condition $(R_1)$. 
Moreover, this fact enables us to claim that 
the Cohen--Macaulayness of $K[\Pc]$ is equivalent to its normality.

Recall that a Noetherian ring $R$ is said to satisfy $(S_n)$ if
\[
\depth R_\fp \ge \min\{ n, \dim R_{\fp}\}
\]
for all $\fp \in \Spec(R)$,
and satisfy $(R_n)$ if $R_\fp$ is a regular local ring for all $\fp \in \Spec(R)$ with
$\dim R_\fp \le n$. The conditions $(S_n)$ and $(R_n)$ are called Serre's conditions. 

The well-known criterion for normality of a Noetherian ring,  
Serre's Criterion (cf. {\cite[Theorem 2.2.22]{BH}}), says that 
a Noetherian ring is normal if and only if it satisfies $(R_1)$ and $(S_2)$.

%\begin{Definition}{\em 
%A noetherian ring $R$ satisfies \emph{Serre's condition} $(R_1)$ 
%if the localization $R_{{\mathfrak p}}$ is a regular local ring 
%for every prime ideal ${\mathfrak p}$ of height $1$.
%}\end{Definition}

We use the following combinatorial criterion of $(R_1)$, which can be found 
in \cite[Exercises 4.15 and 4.16]{brunsgubel}. 
%As a notation, we write $gp(X)$ for the lattice generated by a set $X \subset \ZZ^{d+1}$. 
\begin{Proposition}[\cite{brunsgubel}]\label{prop:r1ex}
Let $M$ be an affine monoid, $K$ a field and $K[M]$ its semigroup $K$-algebra. 
Then $K[M]$ satisfies $(R_1)$ if and only if every facet $\Fc$ of $M$ satisfies the following two conditions: 
\begin{enumerate}
\item $\ZZ(M \cap \Fc) = \ZZ M \cap \Hc$, where $\Hc$ is the supporting hyperplane of $\Fc$; 
\item there exists $x \in M$ such that $\sigma_{\Fc}(x) = 1$, 
where $\sigma_{\Fc}$ is a support form of $\Fc$ with integer coefficients. 
\end{enumerate}
\end{Proposition}

Using this, we can prove 
\begin{Proposition}\label{R1condition}
Let $\Pc$ be an integral cyclic polytope. 
Then $K[\Pc]$ always satisfies the condition $(R_1)$. 
\end{Proposition}
\begin{proof}
First, note that the facets of $\Pc^*$ are in bijection with the facets of the monoid $\ZZ_{\geq 0} \Ac_{\Pc}$. 
%generated by the integer points in $\Pc^*$. 
Let $\Fc$ be a facet of $\Pc^*$ with vertices $v_{i_1}, \dotsc, v_{i_d}$, where $i_1 < \dotsc < i_d$.
Using the same construction as in the proof of Lemma \ref{lemma:basis} (\cite[Lemma 1.6]{HHKO}), 
we get a family $c_j \defa \sum_{l=j}^{d} b_{i_l\dotsc i_{d}}$ of integer points in $\Fc$ 
that is part of a basis of $\ZZ^{d+1}$. 
This implies that every element $x \in \ZZ^{d+1} \cap \Hc$ can be written as a $\ZZ$-linear combination of them.
Therefore, the first condition of Proposition \ref{prop:r1ex} follows.

For the second condition, pick any vertex $v_k$ of $\Pc^*$ that is not in $\Fc$. 
Consider the set $S := \set{k, i_1,\dotsc,i_d} \subset [n]$ with its natural ordering. 
If the position of $k$ in $S$ is even
(i.e., if there is an odd number of $j$ such that $i_j < k$), 
then let $F \subset S$ be the set of elements of odd position. 
Otherwise (i.e., if the position of $k$ in $S$ is odd), let $F$ be the set of elements of even position. 
In any case, $k \notin F$. We write $F = \set{j_1,\dots,j_r}$. We want to do a similar construction to the one above, 
but this time we need to analyse it more closely. 
Consider the vector \[ x' \defa \sum_{l=1}^r \bvec{j_l\dotsc j_r} \,.\] 
By the reasoning above, we know that this is an integer point in $\Fc$, 
but we claim that it has the additional property that the coefficient of each $v_{j_s}$ is strictly positive. 
Indeed, if $s$ is an odd number, then the coefficient of $v_{j_s}$ is an alternating sum of non-increasing values, 
starting and ending with a positive value. Thus it is positive and we only need to consider the case that $s$ is even.  
For this, we compute the coefficient of $v_{j_s}$ in $x'$: 
\[ 
\sum_{l=1}^s \frac{1}{\prod_{\substack{m=l \\m\neq s}}^r \Delta_{j_s j_m}} =
\sum_{l=1}^s \frac{(-1)^{l+1}}{|\prod_{\substack{m=l \\m\neq s}}^r \Delta_{j_s j_m}|} =
\sum_{\substack{l=1\\l \text{ even}}}^s \frac{1}{|\prod_{\substack{m=l\\m\neq s}}^r \Delta_{j_s j_m}|} 
\left( 1- \frac{1}{|\Delta_{j_s j_{l-1}}|} \right). \]
By our choice of $F$, for every two indices in $s_1 < s_2$ in $F$, there is an index in $s_3 \in S$ 
between them $s_1 < s_3 < s_2$. Thus every $\Delta_{j_qj_{q'}}$ in above formula is at least $2$. 
Hence the coefficient of $v_{j_s}$ cannot be zero. Now we define
\[ x \defa x' \pm \bvec{S}, \]
where the sign is ``$+$'' if the position of $k$ in $S$ is odd and ``$-$'' if it is even. 
This ensures that $\sigma_{\Fc}(x) = 1$. It remains to show that $x$ is contained in $\Pc^*$, 
that is that the coefficients of all $v_i, i\in S$ are nonnegative. 
Now for $i \in S \setminus F$, the coefficient of $v_{i}$ is positive by construction. 
For $i \in F$, the coefficient in $x'$ is positive and thus at least $|\prod_{j\in F\setminus\set{i}} \Delta_{ij}|^{-1}$. 
But the coefficient in $\bvec{S}$ is $-|\prod_{j\in S\setminus\set{i}} \Delta_{ij}|^{-1}$, 
so their sum (i.e., the coefficient in $x$) is nonegative, because $F \subset S$. 
\end{proof}

As a consequence of this proposition, we obtain 
\begin{Theorem}\label{CM}
Let $\Pc$ be an integral cyclic polytope and 
$K[\Pc]$ its associated semigroup $K$-algebra. 
Then the following conditions are equivalent: 
\begin{enumerate}
\item $K[\Pc]$ is normal; 
\item $K[\Pc]$ is Cohen--Macaulay; 
\item $K[\Pc]$ satisfies $(S_2)$. 
\end{enumerate}
\end{Theorem}
\begin{proof}
By Hochster's Theorem (see, e.g., \cite[Theorem 6.10]{brunsgubel}), normality implies Cohen--Macaulayness. 
Moreover, Serre's Criterion states that 
normality is equivalent to Serre's conditions $(R_1)$ and $(S_2)$. 
On the other hand, Cohen--Macaulayness implies $(S_2)$, see \cite[p. 62]{BH}, and thus the claim follows. 
\end{proof}

\begin{Remark}{\em 
Using the same methods as employed above, one can also prove that 
an integral cyclic polytope is normal if and only if it is seminormal. 
See \cite[p. 66]{brunsgubel} for the definition and basic properties of seminormality. 
We use the notation from that book. Now, assume that $\Pc$ is not normal. 
Then there exists a point $m$ in $\RR_{\geq 0} \Ac_{\Pc} \cap \ZZ_{\geq 0} \Ac_{\Pc}$ 
which is not contained in $\ZZ_{\geq 0} \Ac_{\Pc}$. 
This point $m$ lies in the interior of a unique face $\Fc$ of $\ZZ_{\geq 0} \Ac_{\Pc}$. 
But using the same construction as above, 
we can show that $\ZZ(\ZZ_{\geq 0} \Ac_{\Pc} \cap \Fc) = \ZZ^{d+1} \cap \Hc$, 
where $\Hc$ is the linear subspace spanned by $\Fc$. Thus $m \in \ZZ(\ZZ_{\geq 0} \Ac_{\Pc} \cap \Fc)$ is an exceptional point, 
and therefore $(\ZZ_{\geq 0} \Ac_{\Pc} \cap \Fc)_*$ is not normal. Hence, $\Pc$ is not seminormal. 
}\end{Remark}

%% file: Gor_2.tex
\section{When is $K[\Pc]$ Gorenstein ?}

The goal of this section is to characterize completely 
when $K[\Pc]$ is Gorenstein, that is, this section is devoted to proving 
\begin{Theorem}\label{Gor}
Let $\Pc=C_d(\tau_1,\ldots,\tau_d)$ be an integral cyclic polytope and 
$K[\Pc]$ its associated semigroup $K$-algebra. 
Then $K[\Pc]$ is Gorenstein if and only if $d=2$, $n=3$ and 
$$(\Delta_{12},\Delta_{23})=(1,2) \;\; \text{{\em or}} \;\; (2,1).$$ 
\end{Theorem}

Thus, by Proposition \ref{equiv}, 
there is essentially only one case where $K[\Pc]$ is Gorenstein. 

\smallskip

Before giving a proof, we prepare the following. 
\begin{itemize}
\item Let 
\begin{eqnarray*}
(v_1,\ldots,v_{d+1})=
\begin{pmatrix}
1      &1           &\cdots                 &\cdots &1                            \\
0      &\Delta_{12} &\Delta_{13}            &\cdots &\Delta_{1,d+1}               \\
\vdots &\ddots      &\Delta_{13}\Delta_{23} &\cdots &\Delta_{1,d+1}\Delta_{2,d+1} \\
\vdots &            &\ddots                 &\ddots &\vdots                       \\
0      &\cdots      &\cdots                 &0      &\prod_{k=1}^d\Delta_{k,d+1}   
\end{pmatrix}
\in \ZZ^{(d+1) \times (d+1)} 
\end{eqnarray*}
and $\Pc^*=\con(\set{v_1,\ldots,v_{d+1}})$. 
\item Let 
\begin{eqnarray*}
&&\ab_1=\left(0,\prod_{j=3}^{d+1}\Delta_{1,j}, - \prod_{j=4}^{d+1} \Delta_{1,j}, \ldots, 
(-1)^d \Delta_{1,d+1}, (-1)^{d+1} \right) \in \ZZ^{d+1}\;\;\;\;\;\text{and} \\
&&\ab_i=\left(\underbrace{0,\ldots,0}_{i-1},\prod_{j=i+1}^{d+1}\Delta_{i,j}, 
-\prod_{j=i+2}^{d+1}\Delta_{i,j},\ldots,(-1)^{d+i-2}\Delta_{i,d+1},(-1)^{d+i-1} \right) \in \ZZ^{d+1} 
\end{eqnarray*}
for $i=2,\ldots,d+1$. 
\item Let $\Hc_i$ be the closed half space in $\RR^{d+1}$ defined by the inequalities 
\begin{eqnarray*}
&&\langle \ab_1 , \xb \rangle \leq \prod_{j=2}^{d+1} \Delta_{1,j}, \;\;\;\;\; \text{for} \;\; i=1, \\
&&\langle \ab_i , \xb \rangle \geq 0, \;\quad\quad\quad\quad \text{for} \;\; i=2,\ldots,d+1, 
\end{eqnarray*}
where $\xb=(x_0,x_1,\ldots,x_d) \in \RR^{d+1}$ and $\langle \ab_i , \xb \rangle$ 
stands for the usual inner product in $\RR^{d+1}$. %each $\Hc_i$ is a supporting hyperplane of each $\Fc_i$ and 
\item By using the above, we have 
\begin{eqnarray}\label{aaa}
\Pc^*= \bigcap_{i=1}^{d+1} \Hc_i \cap \{\xb \in \RR^{d+1} : x_0=1 \}.
\end{eqnarray}
%$\Fc_i=\con(\set{v_1,\ldots,v_{d+1}} \setminus \set{v_i})$ for $1 \leq i \leq d+1$. Let
\end{itemize}

A proof of \eqref{aaa} is given by elemtary computations. 
This establishes an explicit description of the supporting hyperplanes of 
an integral cyclic polytope with $n=d+1$, i.e., a simplex case.

\begin{proof}[Proof of Theorem \ref{Gor}]
First, we can check easily that $K[\Pc]$ is Gorenstein 
when $\Pc=C_2(\tau_1,\tau_2,\tau_3)$ with 
$(\Delta_{12},\Delta_{23})=(1,2)$ or $(\Delta_{12},\Delta_{23})=(2,1)$. 

Thus, what we must do is to show that $K[\Pc]$ is never Gorenstein in other cases. 
Mostly, we concentrate on the case where $\Pc$ is a simplex. 

{\bf The first step.} 
Suppose that $K[\Pc]$ is not normal. 
Then, from Theorem \ref{CM}, $K[\Pc]$ is not Cohen--Macaulay. 
In particular, $K[\Pc]$ cannot be Gorenstein. 

Hence, in the remaining parts, we assume that $K[\Pc]$ is normal. 
Since $\ZZ \Ac_{\Pc}=\ZZ^{d+1}$ by Lemma \ref{lemma:basis}, 
we notice that $K[\Pc]$ is nothing but the Ehrhart ring of $\Pc$ 
when $K[\Pc]$ is normal (cf. \cite[pp. 275--278]{BH}). 
In addition, it is neccesary for the Ehrhart ring $K[\Pc]$ to be Gorenstein that 
$\Pc$ contains only one integer point in its relative interior 
when $\Pc \setminus \partial \Pc \not= \emptyset$. (See, e.g., \cite{DeNegriHibi}.) 
In the following, we verify that there is no such $(\tau_1,\ldots,\tau_n)$. 

{\bf The second step.} 
Assume that $d=2$ and let us consider when $n=3$. 
Suppose that $(\Delta_{12},\Delta_{23})$ is neither $(2,1)$ nor $(1,2)$. 
From Proposition \ref{equiv}, we may assume that $\Delta_{12} \geq \Delta_{23}$. 
When $(\Delta_{12},\Delta_{23})=(1,1)$, we can check that $\Pc$ is not Gorenstein. 
Hence, we assume that either $\Delta_{12} \geq \Delta_{23} \geq 2$ 
or $\Delta_{12} \geq 3$ and $\Delta_{23}=1$ is satisfied. 
Recall from the above statements that 
\begin{eqnarray*}
\Hc_1: \Delta_{13}x_1-x_2 \leq \Delta_{12}\Delta_{13}, \;\;\; 
\Hc_2: \Delta_{23}x_1-x_2 \geq 0, \;\;\;
\Hc_3: x_2 \geq 0. 
\end{eqnarray*}
Then it is enough to that there exist at least two integer points $(1,p_1),(1,p_2) \in \ZZ^3$ such that 
$$\langle (\Delta_{13},-1), p_i \rangle < \Delta_{12}\Delta_{13}, \; 
\langle (\Delta_{23},-1), p_i \rangle > 0 \;\text{and}\; \langle (0,1), p_i \rangle < 0 \;\;\; 
\text{for} \;\; i=1,2.$$ 
\begin{itemize}
\item When $\Delta_{12} \geq \Delta_{23} \geq 2$, 
the integer points $(1,1,1)$ and $(1,2,2)$ are contained 
in $\Pc^* \setminus \partial \Pc^*$. In fact, 
\begin{eqnarray*}
&&\Delta_{13}-1 < \Delta_{13} < \Delta_{12}\Delta_{13}, \;\;\; 
\Delta_{23} - 1 >0, \;\;\; 1>0, \\
&&\Delta_{13}-2 < \Delta_{13} < \Delta_{12}\Delta_{13}, \;\;\; 
2\Delta_{23}-2 > 0, \;\;\; 2>0. 
\end{eqnarray*}
%\begin{align*}
%&&\Delta_{13}-1 < \Delta_{13} \leq \Delta_{12}\Delta_{23}, \;\; 
%&&\Delta_{13}-2 < \Delta_{13} \leq \Delta_{12}\Delta_{23}, \\
%&&\Delta_{23} - 1 >0, &&2\Delta_{23}-2 > 0, \\
%&&1>0, &&2>0. 
%\end{align*}
\item When $\Delta_{12} \geq 3$ and $\Delta_{23} = 1$, 
the integer points $(1,2,1)$ and $(1,3,1)$ are contained in the interior. In fact, 
\begin{eqnarray*}
&&2\Delta_{13}-1 < 2\Delta_{13} < \Delta_{12}\Delta_{13}, \;\;\; 
2\Delta_{23} - 1 >0, \;\;\; 1>0, \\
&&3\Delta_{13}-1 < 3\Delta_{13} \leq \Delta_{12}\Delta_{13}, \;\;\; 
3\Delta_{23}-1 > 0, \;\;\; 1>0. 
\end{eqnarray*}
%\begin{align*}
%&&2\Delta_{13}-1 < 2\Delta_{13} \leq \Delta_{12}\Delta_{23}, \;\; 
%&&3\Delta_{13}-1 < 3\Delta_{13} \leq \Delta_{12}\Delta_{23}, \\
%&&2\Delta_{23} - 1 >0, &&3\Delta_{23}-1 > 0, \\
%&&1>0, &&1>0. 
%\end{align*}
\end{itemize}
Thus, $\Pc$ is not Gorenstein when $n=3$ 
except the case where $(\Delta_{12},\Delta_{23})=(2,1)$ or $(1,2)$. 

%n \geq 4
When $n=4$ and $(\Delta_{12},\Delta_{23},\Delta_{34})=(1,1,1)$, 
then we can also check that $\Pc$ is not Gorenstein. 
Moreover, when $n=4$ and there is at least one $1 \leq i \leq 3$ with $\Delta_{i,i+1} \geq 2$, 
since either $\tau_3-\tau_1 \geq 2$ and $\tau_4-\tau_3 \geq 2$ 
or $\tau_3-\tau_1 \geq 3$ and $\tau_4-\tau_3 = 1$ are satisfied, 
$\Pc'=C_2(\tau_1,\tau_3,\tau_4)$ has at least two integer points 
in $\Pc' \setminus \partial \Pc' \subset \Pc \setminus \partial \Pc$ as discussed above, 
which implies that $\Pc$ is not Gorenstein. 
Similarly, when $n \geq 5$, since $\tau_4-\tau_1 \geq 3$ and $\tau_5-\tau_4 \geq 1$, 
$\Pc$ is not Gorenstein.

{\bf The third step.} 
%d=3  n=4
Assume that $d=3$ and let us consider the case where $n=4$. 
When $(\Delta_{12},\Delta_{23},\Delta_{34})=(1,1,1)$, 
we can check $\Pc$ is not Gorenstein. Thus, we assume that 
there is at least one $1 \leq i \leq 3$ with $\Delta_{i,i+1} \geq 2$. Recall that 
\begin{align*}
&&\Hc_1: \Delta_{13}\Delta_{14}x_1-\Delta_{14}x_2+x_3 \leq \Delta_{12}\Delta_{13}\Delta_{14}, \;\; 
&&\Hc_2: \Delta_{23}\Delta_{24}x_1-\Delta_{24}x_2+x_3 \geq 0, \\
&&\Hc_3: \Delta_{34}x_2-x_3 \geq 0, \;\; 
&&\Hc_4: x_3 \geq 0. 
\end{align*}
\begin{itemize}
\item When $\Delta_{23} \geq 2$, the integer points 
$(1,\Delta_{12}+1,\Delta_{13}+1,q)$, where $q=1$ and 2, are contained in 
$\Pc^* \setminus \partial \Pc^*$. In fact, 
\begin{eqnarray*}
&&\Delta_{13}\Delta_{14}(\Delta_{12}+1)-\Delta_{14}(\Delta_{13}+1)+q
=\Delta_{12}\Delta_{13}\Delta_{14}-\Delta_{14}+q < \Delta_{12}\Delta_{13}\Delta_{14}, \\
&&\Delta_{23}\Delta_{24}(\Delta_{12}+1)-\Delta_{24}(\Delta_{13}+1)+q 
\geq \Delta_{12}\Delta_{24}-\Delta_{24}+q > 0, \\
&&\Delta_{34}(\Delta_{13}+1)-q > 0, \;\;\; q >0. 
\end{eqnarray*}
%where $q$ is 1 or 2. 
\item When $\Delta_{23}=1$ and $\Delta_{12} \geq 2$ and $\Delta_{34} \geq 2$, 
the integer points $(1,2,2,q)$, where $q=1$ and 2, are contained in the interior. In fact, 
\begin{eqnarray*}
&&2\Delta_{13}\Delta_{14}-2\Delta_{14}+q=2\Delta_{12}\Delta_{14}+q < \Delta_{12}\Delta_{13}\Delta_{14}, \\
&&2\Delta_{23}\Delta_{24}-2\Delta_{24}+q =q >0, \;\;\;\;\; 2\Delta_{34}-q > 0, \;\;\;\;\;\; q>0. 
\end{eqnarray*}
%where $q$ is 1 or 2. 
\item When $\Delta_{12} \geq 2$ and $\Delta_{23}=\Delta_{34}=1$, 
the integer points $(1,\Delta_{12},\Delta_{12},1)$ and $(1,\Delta_{12}+1,\Delta_{12}+2,3)$ 
are contained in the interior. In fact, 
\begin{eqnarray*}
&&\Delta_{12}\Delta_{13}\Delta_{14}-\Delta_{12}\Delta_{14}+1<\Delta_{12}\Delta_{13}\Delta_{14}, \\
&&\Delta_{12}\Delta_{23}\Delta_{24}-\Delta_{12}\Delta_{24}+1=1>0, \;\;\; 
\Delta_{12}\Delta_{34}-1=\Delta_{12}-1>0, \;\;\; 1>0
\end{eqnarray*}
and 
\begin{eqnarray*}
&&\Delta_{13}\Delta_{14}(\Delta_{12}+1)-\Delta_{14}(\Delta_{12}+2)+3
=\Delta_{12}\Delta_{13}\Delta_{14}-\Delta_{14}+3<\Delta_{12}\Delta_{13}\Delta_{14}, \\
&&\Delta_{23}\Delta_{24}(\Delta_{12}+1)-\Delta_{24}(\Delta_{12}+2)+3=-2\Delta_{24}+3>0, \\ 
&&\Delta_{34}(\Delta_{12}+2)-3=\Delta_{12}-1>0, \;\;\; 3>0. 
\end{eqnarray*}
\end{itemize}
Thus, $\Pc$ is not Gorenstein when $n=4$. 
Remark that we need not consider the case where $\Delta_{34} \geq 2$ and $\Delta_{12}=\Delta_{23}=1$ 
because of Proposition \ref{equiv} again. 

%n \geq 5
On the other hand, when $n \geq 5$, let $\Pc'=C_3(\tau_1,\tau_3,\tau_4,\tau_5)$. 
Since $\tau_3-\tau_1 \geq 2$, there exist at least two integer points in 
$\Pc' \setminus \partial \Pc' \subset \Pc \setminus \partial \Pc$, 
which means that $\Pc$ is not Gorestein.

{\bf The fourth step.} 
Assume that $d \geq 4$ and $d$ is even. Let us consider 
\begin{eqnarray*}
\alpha_q=(1,\Delta_{12}+1,\Delta_{13}+1,\ldots,\Delta_{1,d-1}+1,\Delta_{1,d},q) \in \ZZ^{d+1} 
%&&\alpha_2=(1,\Delta_{12}+1,\Delta_{13}+1,\ldots,\Delta_{1,d-1}+1,\Delta_{1,d},2) \in \ZZ^{d+1}. 
\end{eqnarray*}
for $q=1$ and 2. We show that $\alpha_1$ and $\alpha_2$ are contained in $\Pc^* \setminus \partial \Pc^*$. 

Now, we have 
\begin{align*}
\langle \ab_1, \alpha_q \rangle &= 
\prod_{j=2}^{d+1}\Delta_{1,j}+\prod_{j=3}^{d+1}\Delta_{1,j} - 
\left(\prod_{j=3}^{d+1} \Delta_{1,j}+\prod_{j=4}^{d+1} \Delta_{1,j} \right) + \cdots \\
&\quad\quad\quad\quad\quad +(-1)^{d-1}\left( \prod_{j=d-1}^{d+1}\Delta_{1,j}+\prod_{j=d}^{d+1}\Delta_{1,j} \right) 
+(-1)^d \prod_{j=d}^{d+1}\Delta_{1,j}+(-1)^{d+1}q \\
&=\prod_{j=2}^{d+1}\Delta_{1,j}+(-1)^{d+1}q=\prod_{j=2}^{d+1}\Delta_{1,j}-q < \prod_{j=2}^{d+1}\Delta_{1,j}, 
\end{align*}\begin{align*}
\langle \ab_i, \alpha_q \rangle &= \Delta_{1,i}\prod_{j=i+1}^{d+1}\Delta_{i,j}+\prod_{j=i+1}^{d+1}\Delta_{i,j}
-\left( \Delta_{1,i+1}\prod_{j=i+2}^{d+1}\Delta_{i,j}+\prod_{j=i+2}^{d+1}\Delta_{i,j} \right)+\cdots+ \\
& (-1)^{d+i-3}\left( \Delta_{1,d-1}\prod_{j=d}^{d+1}\Delta_{i,j}+\prod_{j=d}^{d+1}\Delta_{i,j} \right)+
(-1)^{d+i-2}\Delta_{1,d}\Delta_{i,d+1}+(-1)^{d+i-1}q \\
&=\Delta_{1,i}\prod_{j=i+1}^{d+1}\Delta_{i,j}+
\sum_{k=i}^{d-1}(-1)^{i+k-2}\left( \prod_{j=k+1}^{d+1}\Delta_{i,j}-\Delta_{1,k+1}\prod_{j=k+2}^{d+1}\Delta_{i,j} \right) 
+(-1)^{d+i-1}q \\
&=\Delta_{1,i}\prod_{j=i+1}^{d+1}\Delta_{i,j}+
\sum_{k=i}^{d-1}(-1)^{i+k-1} \left( \Delta_{1,i}\prod_{j=k+2}^{d+1}\Delta_{i,j} \right) +(-1)^{d+i-1}q \\
&=\Delta_{1,i} \left( \prod_{j=i+1}^{d+1}\Delta_{i,j}-\prod_{j=i+2}^{d+1}\Delta_{i,j} \right) + 
\Delta_{1,i} \left( \prod_{j=i+3}^{d+1}\Delta_{i,j}-\prod_{j=i+4}^{d+1}\Delta_{i,j} \right) + \cdots + \\
&\quad\quad\quad\quad\quad 
\Delta_{1,i}\left( \prod_{j=d-1}^{d+1}\Delta_{i,j}-\prod_{j=d}^{d+1}\Delta_{i,j} \right) 
+\Delta_{1,i}\Delta_{i,d+1}-q>0 
\end{align*}
when $i$ is even and 
\begin{align*}
\langle \ab_i, \alpha_q \rangle &= 
\Delta_{1,i} \left( \prod_{j=i+1}^{d+1}\Delta_{i,j}-\prod_{j=i+2}^{d+1}\Delta_{i,j} \right) 
+ \cdots + \Delta_{1,i}\left( \Delta_{i,d}\Delta_{i,d+1}-\Delta_{i,d+1} \right) +q>0 
\end{align*}
when $i$ is odd. 

{\bf The fifth step.} 
Assume that $d \geq 5$ and $d$ is odd. Let us consider 
$$\beta_q=(1,\Delta_{12}+1,\Delta_{13}+1,\ldots,\Delta_{1,d}+1,\Delta_{1,d+1}-q) \in \ZZ^{d+1} $$ 
for $q=1$ and 2. 
Similar to the fourth step, it is easy to see that 
$$\langle \ab_1,\beta_q \rangle < \prod_{j=2}^{d+1}\Delta_{1,j} \;\;\;\text{and}\;\;\; 
\langle \ab_i,\beta_q \rangle > 0 \;\;\;\text{for}\;\;\; i=2,\ldots,d+1.$$ 
In other word, both $\beta_1$ and $\beta_2$ are contained in the interior, as desired. 
\end{proof}

%% file: without_int_ver3.tex
\section{The semigroup ring associated only with vertices of a cyclic polytope}

Throughout this section, $Q$ denotes the affine semigroup $Q_d(\tau_1,\dots ,\tau_n)$.
In this section, we study the normality of the semigroup $K$-algebra $K[Q]$ associated only with the vertices of an integral cyclic polytope. 

\smallskip

Let $S=K[x_1,\dots ,x_n]$ be the polynomial ring over a field $K$.
Let $I_Q$ be the kernel of the surjective ring homomorphism $S \to K[Q]$ sending each $x_i$ to
$t^{v_i}$.
The ideal $I_Q$ is just the toric ideal associated with the matrix \eqref{eq:matrix1}. 
In particular, it is homogeneous with respect to the usual $\ZZ$-grading on $S$.
Recall that the matrix \eqref{eq:matrix1} can be transformed into the form \eqref{eq:mat}. 

\smallskip

By Proposition \ref{thm:prop_cycl_poly} (1), $K[Q]$ is regular when $n = d+1$
and in particular is normal.
When $d =1$, the matrix \eqref{eq:matrix1} transformed as is stated above is of the following form: 
\begin{equation}\label{eq:d=1}
\begin{pmatrix}
1 & 1 & \cdots & 1 \\
0 & \tD_{1,2} & \cdots & \tD_{1,n}
\end{pmatrix}.
\end{equation}
Since $I_Q$ is preserved even if we divide a common divisor of $\tD_{1,2},\dots ,\tD_{1,n}$
out of the second row,
we may assume the greatest common divisor of $\tD_{1,2}, \cdots, \tD_{1,n}$ is equal to $1$.
The ideal $I_Q$ is a defining ideal of a projective monomial curve in $\PP^{n-1}$,
and it is well known (cf. \cite{CN}) that 
the corresponding curve is normal if and only if it is a rational normal curve of degree $n-1$, 
that is, $\tD_{1,i} = i-1$ for all $i-1$ with $2 \le i \le n$
(after the above transformation and re-setting each $\tD_{1,i}$).
Consequently, in the case $d = 1$, the ring $K[Q]$ is normal if and only if
$\tau_2 - \tau_1 = \tau_3- \tau_2 = \cdots = \tau_n - \tau_{n-1}$.

We will show that $K[Q]$ is never normal if $d \ge 2$ and $n = d + 2$.
Our strategy is to make use of the following criterion.

\begin{Lemma}[Ohsugi-Hibi (cf. {\cite[Lemma 6.1]{OH}})]\label{thm:non-normality}
Let $R$ be a toric ring such that the corresponding toric ideal $I$ is
homogeneous.
Suppose $I$ has a minimal system of binomial generators that
contains a binomial consisting of non-squarefree monomials.
Then $R$ is not normal.
\end{Lemma}

Set $\Gamma := \Gamma_d(\tau_1,\dots, \tau_n)$ 
(see Section 1 for the definition of $\Gamma_d(\tau_1,\dots ,\tau_n)$).
Note that there is a one-to-one correspondence between
the faces of $\Gamma$ and the proper faces of $\RR_{\ge 0}Q$;
a subset $W \subseteq [n]$ is a $(d-1)$-dimensional face of $\Gamma$ if and only if
$\RR_{\ge 0} \cdot \setb{v_i}{i \in W}$ is a $d$-dimensional face of $\RR_{\ge 0}Q$.
In the sequel, we tacitly use this correspondence.

If $n = d + 2$, then $I_Q$ is principal, and we can determine
the supports of both monomials appearing in the binomial generator of $I_Q$.
Following the usual convention, we set
$\supp(u) := \setb{i \in [n]}{x_i \mid u}$.

\begin{Lemma}\label{thm:n=d+2}
Assume $n = d + 2$. Then
$K[Q] \cong S/(u - v)$ for some monomials $u, v \in S$ such that
$\supp(u) = \setb{i \in [n]}{i \text{ is odd}}$ and
$\supp(v) = \setb{i \in [n]}{i \text{ is even}}$.
\end{Lemma}
\begin{proof}
Since the rank of $\ZZ Q$ is equal to $\dim \RR_{\ge 0} C_d(\tau_1,\dots ,\tau_n) = d + 1$,
the kernel of the $\QQ$-linear map defined by \eqref{eq:matrix1} is
of dimension $1$. It is then clear that $I_Q$ is principal.
Choose a generator $u-v$ of $I_Q$.
Obviously $\supp(u) \cap \supp(v) = 0$.
Moreover neither $\supp(u)$ nor $\supp(v)$ is a face of $\Gamma$.
Indeed, by the choice of $u - v$,
\begin{equation}
\sum_{i \in \supp(u)} a_i v_i = \sum_{j \in \supp(v)} b_j v_j \tag{$*$}
\end{equation}
for some positive integers $a_i,b_j$, and hence if one of $\supp(u)$ and $\supp(v)$ is a face of $\Gamma$,
say $W$, then the corresponding cone $\RR_{\ge 0}W$ of $\RR_{\ge 0}Q$
contains all the $v_i$ and $v_j$ appearing in $(*)$.
This implies $(*)$ is just a relation among vertices in $\RR_{\ge 0}W$,
which contradicts (1) of Proposition~\ref{thm:prop_cycl_poly}.
Since $n = d+2$, applying (1) of Proposition~\ref{thm:prop_cycl_poly} again,
it follows from $(*)$ that $\supp(u) \cup \supp(v) = [n]$.
Thus $\supp(u)$ and $\supp(v)$ give a partition of $[n]$ by {\em non-faces} of $\Gamma$,
i.e., subsets of $[n]$ which are not in $\Gamma$.

Without loss of generality, we may assume that $1 \in \supp(u)$.
Set
$$
\Lambda := \setb{(F, G) \in (2^{[n]} \setminus \Gamma) \times (2^{[n]} \setminus \Gamma)}{%
1 \in F, F \cap G = \emptyset, F \cup G = [n]}.
$$
Then $(\supp(u), \supp(v)) \in \Lambda$.
On the other hand, the pair $(U,V)$, where
$U := \setb{i \in [n]}{i \text{ is odd}}$ and $V := \setb{i \in [n]}{i \text{ is even}}$,
also belongs to $\Lambda$; indeed, $U$ and $V$ does not satisfy the condition in (3) of Proposition~\ref{thm:prop_cycl_poly}.
Thus what we have only to show to complete the proof is $\#\Lambda = 1$. 
Note that by \cite[Proposition 5.1]{L}, $\Gamma$ is combinatorially equivalent to
the join of the boundary complexes of two simplexes $\Gamma_1$, $\Gamma_2$.
Hence we may identify $\Gamma$ with $\partial\Gamma_1 * \partial\Gamma_2$ to prove $\#\Lambda = 1$,
and may assume $1 \in F_1$.
It is straightforward to verify that $\Lambda = \set{(F_1,F_2)}$.
\end{proof}

Now we will prove the following.

\begin{Theorem}\label{thm:n=d+2 with d > 1}
Assume $d \ge 2$ and $n = d+2$. Then $K[Q]$ is never normal.
\end{Theorem}
\begin{proof}
Set $U := \setb{i \in [n]}{i \text{ is odd}}$ and $V := \setb{i \in [n]}{i \text{ is even}}$. By Lemma~\ref{thm:n=d+2},
$$
K[Q] \cong S / \left( \prod_{i \in U} x_i^{a_i} - \prod_{j \in V} x_j^{b_j} \right).
$$
Set $u = \prod_{i \in U} x_i^{a_i}$ and $v = \prod_{j \in V} x_j^{b_j}$.
By Lemma \ref{thm:non-normality}, it suffices to show that neither $u$ nor $v$ is squarefree.
Note that the following equality holds.
$$
\begin{pmatrix}
1 & 1 & \cdots & 1 \\
\tau_1 & \tau_2 & \cdots & \tau_n \\
\tau_1^2 & \tau_2^2 & \cdots & \tau_n^2 \\
\vdots & \vdots & \cdots & \vdots \\
\tau_1^d & \tau_2^d & \cdots & \tau_n^d
\end{pmatrix}
\begin{pmatrix}
a_1 \\
-b_2 \\
a_3 \\
-b_4 \\
\vdots
\end{pmatrix}
= \begin{pmatrix}
0 \\
0 \\
\vdots \\
0
\end{pmatrix}
\in \ZZ^{d+1}.
$$
By Lemma \ref{eg:delta},
\begin{equation}\label{eq:eq_matrix}
\begin{pmatrix}
1         & 1          & 1           & \cdots & 1                 & 1           & 1 \\
0         & \tD_{1,2} & \tD_{1,3} & \cdots & \tD_{1,d}   & \tD_{1,d+1}  & \tD_{1,n} \\
0         & 0          & \tD_{2,3}  & \cdots & \tD_{2,d}   & \tD_{2,d+1} & \tD_{2,n} \\
\vdots & \vdots   & \ddots    & \ddots &                   & \vdots    & \vdots \\
\vdots & \vdots   &              & \ddots & \ddots         &\vdots   & \tD_{d - 1,n} \\
0         & 0          & 0           & \cdots & 0                 & \tD_{d,d+1} & \tD_{d,n}
\end{pmatrix}
\begin{pmatrix}
a_1 \\
-b_2 \\
a_3 \\
-b_4 \\
\vdots
\end{pmatrix}
= \begin{pmatrix}
0 \\
0 \\
\vdots \\
0
\end{pmatrix}.
\end{equation}

For a proof by contradiction, suppose either $u$ or $v$ is squarefree. This is equivalent to say that
$\sum_{i \in U} a_i = \#U$ or $\sum_{j \in V} b_j = \#V$.
By the equation \eqref{eq:eq_matrix}, it follows that
\begin{equation}
\sum_{i \in U}a_i = \sum_{j \in V}b_j.
\end{equation}

{\bf The case $d$ is even.} Then $d = 2l$ for some positive integer $l$,
$n = 2l + 2$, $\#U = \#V = l+1$, which implies both of $u$ and $v$ are
squarefree. By the equation \eqref{eq:eq_matrix},
we have $\tD_{d,d+1} = \tD_{d,n}= 0$, while clearly $\tD_{d,n} > \tD_{d,d+1}$ holds,
a contradiction.

{\bf The case $d$ is odd.}  Then $d = 2l -1$ for some integer $l$ with $l > 1$,
$n = 2l + 1$ and $\#U =\#V + 1= l + 1$,
which implies that $v$ cannot be squarefree since $\sum_{i \in U}a_i \ge \#U$.
Thus $u$ is squarefree, that is, $a_i = 1$ for all $i \in U$.
Moreover one of the $b_j$ is $2$ and the others are $1$.
On the other hand, it follows from \eqref{eq:eq_matrix} that
$-\tD_{d, d+1}b_{2l} + \tD_{d,n}a_{2l+1} = 0$.
Since $\tD_{d,d+1} < \tD_{d,n}$, we conclude that $b_{2l} = 2$ and hence
$$
\tD_{d,n} = 2\tD_{d,d+1}.
$$
For simplicity, we set $c_i = a_i$ for odd $i$ and $c_i = -b_i$ for even $i$.
Hence $c_1 = c_3 = \cdots = c_n = 1$, $c_2 = c_4 = \cdots = c_{n-3} = -1$, and $c_{n-1} = -2$.
By the equation \eqref{eq:eq_matrix} again,
\begin{align*}
0 = \sum_{i=2}^n \tD_{1,i}c_i = \sum_{i=2}^n \left(\sum_{j=1}^{i-1}\Delta_{j,j+1}\right) c_i
  = \sum_{j=1}^{n-1} \Delta_{j,j+1} \left( \sum_{i = j+1}^n c_i \right)
\end{align*}
Since $n \ge 5$ by the hypothesis that $n$ is odd and $d \ge 2$,
we may divide the last summation in the above equality as follows.
Set
$$
s_1 := \sum_{j = 1}^{n-3} \left( \sum_{i = j+1}^n c_i \right) \Delta_{j,j+1},
$$
and
$$
s_2 = \Delta_{n-2,n-1}(c_{n-1} + c_n) + \Delta_{n-1,n}c_n = - \Delta_{d,d+1} + \Delta_{d+1,n}
$$
Then $s_1 + s_2 = \sum_{j=1}^{n-1} \Delta_{j,j+1} \left(\sum_{i = j+1}^n c_i \right) = 0$.
An easy observation shows that each coefficient $\sum_{i = j+1}^n c_i$
of $\Delta_{j,j+1}$ in $s_1$ is $0$ if $j$ is even and otherwise negative.
Hence the inequality $s_1 < 0$ follows since $n - 3 \ge 2$.
We will show that $s_2 \le 0$. If this is the case, then $s_1 + s_2 < 0$ holds on the contrary to the fact $s_1 + s_2 = 0$, which completes the proof.

Suppose $s_2 > 0$. Then
$$
\tau_n - \tau _{d+1} = \Delta_{d+1,n} > \Delta_{d,d+1} = \tau_{d+1} - \tau_d,
$$
and hence $\tau_n - \tau_d > 2(\tau_{n-1} - \tau_d)$.
It follows that
\begin{align*}
\tD_{d,n} &= (\tau_n - \tau_d)(\tau_n - \tau_{d-1}) \cdots (\tau_n - \tau_1) \\
            &> 2(\tau_{n-1} - \tau_d)(\tau_{n-1} - \tau_{d-1}) \cdots (\tau_{n-1} - \tau_1)
              = 2\tD_{d,n-1},
\end{align*}
which is absurd.
\end{proof}

As is stated above Lemma~\ref{thm:non-normality}, the $K$-algebra $K[Q]$ is normal if and only if
$\tau_2 - \tau_1 = \tau_3 - \tau_2 = \cdots = \tau_n - \tau_{n-1}$,
when $d = 1$.
Though we do not have a complete answer on the normality of $k[Q]$ when $n > d +2$,
we strongly believe the following holds.

\begin{Conjecture}
The $K$-algebra $K[Q]$ is normal only in one of the following cases: 
\begin{enumerate}
\item $n = d + 1$; 
\item $d = 1$ and $\tau_2 - \tau_1 = \tau_3- \tau_2 = \cdots = \tau_n - \tau_{n-1}$.
\end{enumerate}
\end{Conjecture}

The following proposition tells us that there are a lot of $K[Q]$
which are not normal when $n \ge d + 3$.

\begin{Proposition}\label{thm:GTR d+2}
Assume $n \ge d + 3$.
If $\tD_{d,d+1} \nmid \tD_{d,s}$ for some $s$ with $d+2 \le s \le n$, then $K[Q]$ is not normal.
\end{Proposition}

\begin{proof}
Suppose $Q$ is normal.
Since the subset $\set{1, \dots ,d}$ of $[n]$ satisfies the condition in (3) of Proposition~\ref{thm:prop_cycl_poly},
the cone generated by $v_1, \dots ,v_d$ forms a facet of $\RR_{\ge 0} Q$.
Let $\Fc$ denote this facet.
Then $Q$ together with $\Fc$ satisfies the condition in Proposition~\ref{prop:r1ex}, and in particular,
there exists an element $x \in Q$ such that $\sigma_{\Fc}(x) = 1$,
where $\sigma_{\Fc}$ is  a support form of $\Fc$ with integer coefficients.

We will describe $\sigma_{\Fc}$ explicitly. Let $\Hc$ be the supporting hyperplane of $\Fc$.
Note that we can freely identify $Q$ with the affine semigroup associated with
the matrix in Lemma \ref{eg:delta}.
After this identification, the vector
$\ab_d = (0,\dots ,0,1) \in \ZZ^{d+1}$ defines $\Hc$ as is stated below of
Theorem \ref{Gor}. Thus
$$
\Hc = \setb{x \in \RR^{d+1}}{\langle \ab_d , x\rangle = 0},
$$
and $\langle \ab_d, x \rangle \in \ZZ_{> 0}$ for all $x \in Q \setminus \Fc$.
We set $\ZZ Q_\Hc := \ZZ Q / \ZZ Q \cap \Hc$.
Note that $\ZZ Q_\Hc \cong \ZZ$.
Let $v_0 \in \ZZ Q$ be an element whose image in $\ZZ Q_\Hc$
is a free basis of $\ZZ Q_\Hc$.
Then the support form $\sigma_{\Fc}$ of $Q$ and $\Fc$ is defined as
$$
\sigma_{\Fc}(x) = \frac{\langle \ab_d,x \rangle}{\langle \ab_d, v_0\rangle}
$$
for all $x \in \RR^{d+1}$, and $\sigma_{\Fc}(x) = 0$ for $x \in Q \cap \Fc$
and $\sigma_{\Fc}(x) \in \ZZ_{>0}$ for $x \in Q \setminus \Fc$ (see \cite[Remark 1.72 and p.55]{brunsgubel} for the construction and the property of a support form).
Recall that there exists an element $x \in Q$ such that $\sigma_{\Fc}(x) = 1$.
Since $\Fc$ is generated by $v_1,\dots ,v_d$,
the element $x$ can be written as $x = y + \sum_{i = d+1}^n \lambda_i v_i$ for some
$\lambda_i \in \ZZ_{\ge 0}$ and $y \in Q \cap \Fc$.
By definition, $\langle \ab_d, v_i \rangle = \tD_{d,i}$ for $i = d+1,\dots ,n$,
and $\langle \ab_d ,y \rangle = 0$.
Since $\tD_{d,d+1} < \tD_{d,d+2} < \cdots < \tD_{d,n}$,
it follows that $0 < \langle \ab_d ,v_{d+1} \rangle < \cdots < \langle \ab_d, v_n \rangle$,
and hence
$$
1 = \sigma_{\Fc}(x) \ge (\sum_{i = d+1}^n \lambda_i) \sigma_{\Fc}(v_{d+1}) > 0.
$$
Therefore $\sigma_{\Fc}(v_{d+1}) = 1$ and $x = y + v_{d+1}$.
Thus we can replace $v_0$ by $v_{d+1}$.
However it follows from the fact $v_s \in Q$ that
$$
\frac{\tD_{d,s}}{\tD_{d,d+1}} = \frac{\langle \ab_d, v_s \rangle}{\tD_{d,d+1}}
     = \sigma_{\Fc}(v_s) \in \ZZ,
$$
contrary to the hypothesis $\tD_{d,d+1} \nmid \tD_{d,s}$.
\end{proof}

Since there exists a lot of non-normal $K[Q]$, it is natural to ask when $K[Q]$ is Cohen-Macaulay.
Clearly if $n = d + 2$, then $K[Q]$ is a complete intersection,
and hence in particular, Cohen-Macaulay.
So far, we have never found an example of $K[Q]$ which is Cohen-Macaulay,
in the case $d \ge 2$ and $n > d + 2$. Thus we expect the following 
\begin{Conjecture}
The $K$-algebra $K[Q]$ is never Cohen-Macaulay if $d \ge 2$ and $n > d + 2$.
\end{Conjecture}